\documentclass[12pt]{amsart}
\usepackage{amssymb, amsmath}
\usepackage{amsfonts,amsbsy,amssymb,amsthm}
\usepackage[usenames,dvipsnames]{color}
\usepackage{url}

\newtheorem{theorem}{\bf Theorem}[section]
\newtheorem{proposition}[theorem]{\bf Proposition}
\newtheorem{corollary}[theorem]{\bf Corollary}

\newtheorem{lemma}[theorem]{\bf Lemma}
\newtheorem{prp}[theorem]{\bf Proposition}

\theoremstyle{definition}
\newtheorem{definition}[theorem]{\bf Definition}
\newtheorem{remark}[theorem]{\bf Remark}
\newtheorem{example}[theorem]{\bf Example}

\newcommand{\ba}{\begin{array}}
\newcommand{\ea}{\end{array}}

\def \R{{\mathbb R}}
\def \C{{\mathbb C}}

\def \I{{\mathcal I}}

\def \x{{\mathbf{x}}}

\def \e{{\mathbf{e}}}

\def \H{{\mathbb H}}

\def \N{{\mathbb N}}

\def \Delta{\triangle}

\def \diag{\mathrm{diag}}

\def \rank{\mathrm{rank}}

\def \rank{\mathrm{rank}}
\def \diag{\mathrm{Diag}}

\def \Im{\mathsf{Im}}

\def \herm{\mathsf{Herm}}

\def \x{{\mathbf{x}}}

\def \z{{\mathbf{z}}}

\newenvironment{ex}{\begin{example}\rm}{\hfill$\Box$\end{example}}

\DeclareMathOperator{\inter}{int} % interior
\DeclareMathOperator{\GL}{GL} % general linear group
\DeclareMathOperator{\init}{in}
\DeclareMathOperator{\cl}{cl}
\DeclareMathOperator{\myspan}{span}
\DeclareMathOperator{\cone}{cone}
\DeclareMathOperator{\SO}{SO} % special orthogonal group

\newcommand{\sym}{\mathcal{S}}

\newcommand{\compl}{\mathsf{c}}
\newcommand{\psdtext}{\mathrm{psd}}

\title{Conic stability of polynomials and positive maps}

\author{Papri Dey}

\author{Stephan Gardoll}

\author{Thorsten Theobald}

\address{P.~Dey: Department of Mathematics, University of Missouri,
  Columbia, MO, USA}

\address{S.~Gardoll and T.Theobald: Goethe-Universit\"at, FB 12 -- Institut f\"ur Mathematik,
  Postfach 11 19 32, 60054 Frankfurt am Main, Germany}

\begin{document}

\begin{abstract}
Given a proper cone $K \subseteq \R^n$, a multivariate polynomial
$f \in \C[\mathbf{z}] = \C[z_1, \ldots, z_n]$ is called $K$-stable 
if it does not have a root whose vector of the
imaginary parts is contained in the interior
of $K$. If $K$ is the non-negative orthant, then $K$-stability 
specializes to the usual notion of stability of polynomials.

We study conditions and certificates for the $K$-stability of a 
given polynomial $f$, especially for the case of determinantal 
polynomials as well as for quadratic
polynomials. A particular focus is on psd-stability.
For cones $K$ with a spectrahedral representation, we construct
a semidefinite feasibility problem, which, in the case of feasibility, 
certifies $K$-stability of $f$. This reduction to a semidefinite 
problem builds upon techniques from the connection
of containment of spectrahedra and positive maps.

In the case of psd-stability, if the criterion is satisfied, we can 
explicitly construct a determinantal representation of the given 
polynomial. We also show that under certain conditions, for a $K$-stable
polynomial $f$, the criterion is at least fulfilled for some scaled version
of $K$.
\end{abstract}

\maketitle
\section{Introduction}

Recently, there has been wide-spread research interest in stable polynomials and the
geometry of polynomials, accompanied by a variety of new connections to other
branches of mathematics
(including combinatorics \cite{braenden-hpp}, 
differential equations \cite{borcea-braenden-2010},
optimization \cite{straszak-vishnoi-2017},
probability theory \cite{bbl-2009},
applied algebraic geometry \cite{volcic-2019},
theoretical computer science \cite{mss-interlacing1, mss-interlacing2}
and statistical physics \cite{borcea-braenden-leeyang1}). 
See also the surveys of
Pemantle \cite{pemantle-2012} and Wagner \cite{wagner-2010}.
Stable polynomials are strongly linked to matroid theory \cite{braenden-hpp}, 
as delta-matroids arise from support sets of stable polynomials.

In this paper, we concentrate on the generalized notion of $K$-stability as introduced 
in \cite{joergens-theobald-conic}. Given a proper cone $K \subseteq \R^n$,
a polynomial $f \in \C[\mathbf{z}] = \C[z_1, \ldots, z_n]$ is called \emph{$K$-stable} if
$\mathcal{I}(f)\cap \inter K=\emptyset$, where $\inter K$ is the interior of $K$ and 
$\I(f)$ denotes the imaginary projection of $f$ (as formally defined in
Section~\ref{se:prelim}).
Note that $(\R_{\ge 0})^n$-stability coincides with the usual stability,
and stability with respect to the positive semidefinite cone on the
space of symmetric matrices is denoted as \emph{psd-stability}.
In the case of a homogeneous polynomial, $K$-stability of $f$ is 
equivalent to the containment of $\inter K$ in a hyperbolicity cone of $f$
(see Section~\ref{se:prelim}), which also provides a link to
hyperbolic programming.

Here, we study conditions and certificates
for the $K$-stability of a given polynomial $f \in \C[\mathbf{z}]$,
especially for the case of determinantal polynomials of the form
$f(\mathbf{z}) \ = \ \det(A_0 + A_1 z_1 + \cdots + A_n z_n)$
with symmetric or Hermitian matrices $A_0, \ldots, A_n$ as well as for
quadratic polynomials. A particular focus is on psd-stability.

Specifically, for cones $K$ with a spectrahedral representation we construct
a semidefinite feasibility problem, which, in the case of non-emptiness,
certifies $K$-stability of $f$. This reduction to a semidefinite problem
builds upon two ingredients.
Firstly, we characterize certain conic components in the complement of
the imaginary projection of the (not necessarily homogeneous) polynomial
$f$. Secondly, the sufficient criterion employs techniques
from \cite{ktt-2013} on containment problems of spectrahedra and 
positive maps in order to check whether
$\inter K \subseteq \mathcal{I}(f)^{\compl}$. 
For the special case of usual stability, we will recover the well-known
determinantal stability criterion of Borcea and Br\"anden
(see Proposition~\ref{pr:crit-stable} and
Remark~\ref{rem:usual-stability}) and thus obtain, as a byproduct,
an alternative proof of that statement.

In the case of psd-stability, if the sufficient criterion is satisfied, we can 
explicitly construct a determinantal representation of the given polynomial,
see Corollary~\ref{co:quadr-certificate}.
To this end, the determinantal criterion for psd-stability 
from \cite{joergens-theobald-conic} can be seen as a special case of 
our more general results.
The procedure enables to check and certify the conic stability for a 
large subclass of polynomials.

Moreover, we show that under certain preconditions, 
there always exists a positive scaling factor such that the sufficient
criterion applies to a scaled version of the polynomial (or, 
equivalently, a scaled version of the cone). See Theorem~\ref{th:nu-cont}.

The paper is structured as follows.
Section~\ref{se:prelim} provides relevant background on imaginary projections,
conic stability and determinantal representations.
In Section~\ref{se:coniccomponents}, we study the conic components in the
complement of the imaginary projection for the relevant classes of polynomials.
Section~\ref{se:conicstab-posmaps} develops the sufficient
criterion for $K$-stability based on the techniques from positive maps.
The scaling result is contained in Section~\ref{se:scaled}, and 
Section~\ref{se:conclusion} concludes the paper.

\medskip

\noindent
{\bf Acknowledgments.}
Part of this work was done while the first and the third author were visiting 
the Simons Institute for the Theory of Computing within the
semester program ``Geometry of Polynomials''. They are grateful for the
inspiring atmosphere there.
Thanks to Bernd Sturmfels for encouraging us to work jointly on this topic
and to the anonymous referees for very valuable comments.

The first author would like to gratefully acknowledge the financial support 
through a Simons-Berkeley postdoctoral fellowship and the third author 
the support through DFG grant TH 1333/7-1.

\section{Preliminaries\label{se:prelim}}

Throughout the text,
bold letters will denote $n$-dimensional vectors unless noted otherwise.

\subsection{Imaginary projections and conic stability\label{se:prelim-conicstab}}

For a polynomial $f \in \C[\mathbf{z}]$, define its \emph{imaginary projection} 
$\mathcal{I}(f)$ as the projection of the variety of $f$ onto its 
imaginary part, i.e.,
\begin{equation}
\I(f ) = \{\Im(\z) = (\Im(z_{1}), \dots,\Im(z_{n})) : f(\z) = 0\},
\end{equation}
where $\Im(\cdot)$ denotes the imaginary part of a complex 
number \cite{jtw-2019}.

Let $\sym_d, \sym_d^{+}$ and $\sym_{d}^{++}$ denote the set of symmetric $d \times d$ 
matrices as well as the subsets of positive semidefinite and positive definite matrices. 
Moreover, let $\herm_d$ be the set of all Hermitian $d \times d$-matrices.

We consider the following generalization of stability.
Let $K$ be a \emph{proper} cone in $\R^n$, that is, a full-dimensional, closed 
and pointed convex cone in $\R^n$.

\begin{definition}\label{de:kstable}
        A polynomial $f \in \C[\mathbf{z}]$ is called $K$-\emph{stable}, if $f(\mathbf{z})\neq0$ whenever $\Im(\mathbf{z})\in \inter K$.

  If $f \in \C[Z]$ on the symmetric matrix variables
  $Z = (z_{ij})_{n \times n}$ is $\sym_n^+$-stable, then
  $f$ is called \emph{positive semidefinite-stable} (for short, \emph{psd-stable}).
\end{definition}

A stable or $K$-stable polynomial with real coefficients is called 
\emph{real stable} or \emph{real} $K$-\emph{stable}, respectively.

\begin{remark}
1. A set of the form $\R^n + iC$, where $C$ is an open convex cone, is called
a \emph{Siegel domain (of the first kind)}. Siegel domains provide an important concept
in function theory of several complex variables and harmonic analysis, 
see the books \cite{hoermander-1990, piaetetski-shapiro-1969,stein-weiss-1971}.

2. The \emph{Siegel upper half-space} (or \emph{Siegel upper half-plane}) $\mathcal{H}_g$
of  degree $g$ (or genus $g$) is defined as
\[
  \mathcal{H}_g \ =  \{ A \in \C^{g \times g} \text{ symmetric } \, : \, \Im(A) \text{ is positive definite}
   \} \, ,
\]
where $\Im(A) = (\Im(a_{ij}))_{g \times g}$
(see, e.g., \cite[\S{2}]{van-der-geer-2008}).
The Siegel upper half-space occurs in algebraic
geometry and number theory as the domain of modular forms.
Using that notation, psd-stability can be viewed as stability with respect
to the Siegel upper half-space.
\end{remark}

A form (i.e., a homogeneous polynomial) $f \in \R[\z]$ is \emph{hyperbolic} in direction $\e \in \R^{n}$ if $f(\e)\neq 0$ and for every $\x \in \R^{n}$ the univariate polynomial $t \mapsto f(\x+t \e)$ has only real roots. 
The cone $C(\mathbf{e})=\{\mathbf{x}\in\R^n:f(\mathbf{x}+t\mathbf{e})=0 \Rightarrow t<0\}$ is called the \emph{hyperbolicity cone} of $f$ with respect to $\mathbf{e}$. This cone $C(\mathbf{e})$ is convex,
$f$ is hyperbolic with respect to every point $\mathbf{e}' \in C(\mathbf{e})$
and
$C(\mathbf{e})=C(\mathbf{e}')$ (see \cite{garding-59}).

Let $f$ be a hyperbolic polynomial and $C(\e)$ denote the hyperbolicity cone 
containing $\e$. 
By definition of $K$-stability, a homogeneous polynomial $f$ is hyperbolic w.r.t.\ every 
point $\mathbf{e}' \in C(\e)$ if and only if $f$ is $(\cl C(\e))$-stable,
where $\cl$ denotes the topological closure of a set.
The following theorem in \cite{joergens-theobald-conic} 
reveals the connection between $K$-stable polynomials and hyperbolic polynomials.
\begin{theorem} \label{ThmConicStability}
 For a homogeneous polynomial $f \in \R[\z]$, 
 the following are equivalent.
 \begin{enumerate}
  \item $f$ is $K$-stable.
  \item $\I(f) \cap \inter K = \emptyset$.
  \item $f$ is hyperbolic w.r.t.\ every point in $\inter K$.
 \end{enumerate}
\end{theorem}

By~\cite{joergens-theobald-hyperbolicity},
the hyperbolicity cones of a homogeneous polynomial
$f$ coincide with the components of $\I(f)^\compl$, where 
$\mathcal{I}(f)^{\compl}$ denotes the complement of $\mathcal{I}(f)$. 
This implies:
\begin{corollary}
 A hyperbolic polynomial $f \in \R[\mathbf{z}]$ is $K$-stable 
 if and only if $\inter K \subseteq C(\e)$ for some hyperbolicity 
 direction $\e$ of $f$. 
 \end{corollary}

 \begin{proof} This follows from the observation that a 
   hyperbolic polynomial $f \in \R[\z]$ is $K$-stable 
   if and only if $\inter K \subseteq \I(f)^\compl$.
 \end{proof}

 It is shown in \cite{joergens-theobald-hyperbolicity} that the number of hyperbolicity cones of a homogeneous polynomial $f \in \R[\mathbf{z}]$ is at most $2^d$ for $d \le n$ and at most $2 \sum_{k=0}^{n-1} \binom{d-1}{k}$ for $d > n$.
 
\subsection{Determinantal representations}
A \emph{determinantal polynomial} is a polynomial of the form
$f(\mathbf{z}) = \det(A_0 + \sum_{j=1}^n A_j z_j)$. For our purposes,
we always assume that the matrices $A_0, \ldots, A_n$ are Hermitian unless
stated otherwise. If the constant 
coefficient matrix $A_{0}$ is positive definite or the
identity, then the determinantal polynomial is called \emph{definite} or 
\emph{monic} determinantal polynomial, respectively. Helton, McCullough 
and Vinnikov showed
that every polynomial $p \in \R[\mathbf{z}]$ with $p(0) \neq 0$
has a symmetric determinantal representation of the form
$p(\mathbf{z}) = \det(A_0 + \sum_{j=1}^n A_j z_j)$
with real symmetric matrices $A_0, \ldots, A_n$
(\cite[Theorem 14.1]{hmv-2006}, see also Quarez \cite[Theorem 4.4] {quarez-2012}
and, for the earlier result of a not necessarily symmetric determinantal
representation, Valiant \cite{valiant-1979} and its 
exposition in B\"urgisser et al.\ \cite{bcs-book}).
Note that $A_0$ is not necessarily positive definite and not even 
necessarily positive semidefinite.

In \cite{helton-vinnikov-2007} and \cite{netzer-thom-2012}, it 
was shown that several classes of polynomials have monic determinantal 
representations due to the connection to real zero polynomials. Here,
a polynomial $f\in\R[\mathbf{z}]$ is called \emph{real zero}, 
if the mapping $t\mapsto f(t\cdot \mathbf{z})$ has only real roots.
Br\"and\'{e}n has constructed a real zero polynomial for which $A_0$
cannot be taken to be positive definite in a determinantal
representation \cite{braenden-2011}. 
Recently, Dey and Pillai \cite{dey-pillai-2018} added a complete 
characterization of the quadratic case by also using the connection to 
real zero polynomials.

\begin{prp}[\cite{dey-pillai-2018}]
A quadratic polynomial 
$f(\mathbf{z}) = \mathbf{z}^T A \mathbf{z} + \mathbf{b}^T \mathbf{z} + 1
\in \R[\mathbf{z}]$
is a real zero polynomial if and only if $Q/(1,1) := A - \frac{1}{4} \mathbf{b} \mathbf{b}^T$ is negative
  semidefinite. The polynomial $f(\z)$ has a monic determinantal representation if and only if at least one of the following conditions holds:
\begin{itemize}
\item $A$ is negative semidefinite.
\item $Q/(1,1)$ is negative
  semidefinite and $\rank(Q/(1,1)) \le 3$.
\end{itemize}
\end{prp}

\subsection{Real stable polynomials}
 As specified in the Introduction and Section~\ref{se:prelim-conicstab},
 a real polynomial $f$ is real stable if it is real $K$-stable with respect 
 to the non-negative orthant $K = \R_+^n$. This holds true 
 if and only if for every $\e \in \R^{n}_{>0}$ and $\x \in \R^{n}$, the univariate polynomial $t \mapsto f(t\e+\x)$ is real-rooted. Indeed, a particular prominent class of real stable polynomials is generated from determinantal polynomials as follows.
\begin{prp}\normalfont{(\cite[Thm. 2.4]{borcea-braenden-2008})}
\label{pr:crit-stable}
Let $A_1, \ldots, A_n$ be positive semidefinite $d \times d$-matrices and
$A_0$ be a Hermitian $d \times d$-matrix. Then
\[
  f(\mathbf{z}) \ = \ \det(A_0 + \sum_{j=1}^n A_j z_j)
\]
is real stable or the zero polynomial.
\end{prp}
It is also known a real polynomial $f \in \R[\mathbf{z}]$ is real stable if and only if the (unique) homogenization polynomial w.r.t.\ the variable $z_{0}$ is hyperbolic w.r.t.\ every vector $\e \in \R^{n+1}$ such that $e_{0}=0$ and $e_{j} > 0$ for all $1 \leq j \leq n$ (see \cite{borcea-braenden-2010}).

\begin{example}
The class of homogeneous stable polynomials is contained in the following class of
Lorentzian polynomials,
see \cite{lorentzian-braenden-huh,lorentzian-hehl-itin}.
Let $f \in \R[\mathbf{z}]$ be homogeneous of degree $d \ge 2$ with only positive coefficients. 
$f$ is called \emph{strictly Lorentzian} if 
\begin{itemize}
\item $d = 2$ and the \textit{Hessian} $\H(f)=(\partial_{i} \partial_{j} f)_{i,j=1}^{n}$ 
 is non-singular and has exactly one positive eigenvalue
 (i.e., $\H(f)$ has the \emph{Lorentzian signature} $(1,n-1)$, which expresses
  that $f$ has one positive eigenvalue and $n-1$ negative eigenvalues
  \cite{lorentzian-hehl-itin}),
\item or $d > 2$ and for all $\alpha \in \N_0^n$ with $|\alpha|=d-2$, the 
  $\alpha$-th derivative $\partial^{\alpha} f$ is strictly Lorentzian.
\end{itemize}
By convention, in degrees 0 and 1, every polynomial with only positive 
coefficients is strictly Lorentzian. Limits of strictly Lorentzian polynomials
are called \emph{Lorentzian}.
\end{example}

Concerning psd-stability, the following variant of
Proposition~\ref{pr:crit-stable} is known.

\begin{proposition}\normalfont{(\cite[Thm. 5.3]{joergens-theobald-conic})}
\label{pr:crit-psd-stable}
Let $A= (A_{ij})_{n \times n}$ be a Hermitian block matrix with $n \times n$ blocks
of size $d \times d$. If $A$ is positive semidefinite and $A_0$ a Hermitian
$d \times d$-matrix, then the polynomial
$f(Z) \ = \ \det(A_0 + \sum_{i,j=1}^n A_{ij} z_{ij})$ on the set
of symmetric $n \times n$-matrices is psd-stable or identically zero.
\end{proposition}

Determinantal representations of complex polynomials which are stable with
respect to the unit ball of symmetric matrices have been studied in 
\cite{gkv-2016,gkv-2017}.

In the present paper,
for cones $K$ with a spectrahedral representation, we derive
a semidefinite problem, which, in the case of feasibility,
certifies $K$-stability of $f$. For the case of psd-stability, if that 
criterion is satisfied, we can explicitly
construct the determinantal representation of 
Proposition~\ref{pr:crit-psd-stable}. In this respect, the criterion from
Proposition~\ref{pr:crit-psd-stable} can be seen as a special case of our
treatment.

The following examples serve to pinpoint some 
relationships between stable, psd-stable and determinantal polynomials.
\begin{example} {a) A quadratic determinantal polynomial
does not need to be stable in order to be psd-stable (with respect to a suitable
ordering identification between the variables $z_i$ and the matrix variables
$z_{jk}$). Namely, the determinantal polynomial 
\begin{equation*}
 f(z_{1},z_{2},z_{3}) =(z_{1}+z_{3})^{2}-z_{2}^{2} 
   = (z_1 + z_3 - z_2)(z_1+z_3 +z_2)
\end{equation*}
is not stable, because $(1,2,1) \in \mathcal{I}(f) \cap \R_{>0}^3$.
However, in the matrix variables 
 $Z = \begin{pmatrix}z_1 & z_2 \\
  z_2 & z_3 \end{pmatrix}$, the polynomial $f(Z) = f(z_1,z_2,z_3)$
  is psd-stable. To see this, observe that by the arithmetic-geometric
  mean inequality, every
  $\mathbf{y} \in \mathcal{I}(f) = \{ \mathbf{y} \in \R^3 \, : \,
  y_1+y_3 = y_2 \text{ or } y_1 + y_3 = - y_2\}$
 satisfies
\[
  \det \begin{pmatrix}
    y_1 & y_2 \\
    y_2 & y_3
  \end{pmatrix}
  \ = \
  \det
  \begin{pmatrix}
    y_1 & \pm (y_1 + y_3) \\
    \pm (y_1 + y_3) & y_3
  \end{pmatrix}
  \ = \ y_1 y_3 - (y_1 + y_3)^2
  \ \le \ 0
\]
and thus
$\mathbf{y} \not\in \inter \sym_2^+$.

\smallskip

\noindent
b) An example of a non-psd-stable determinantal polynomial on $2 \times 2$-matrices,
  i.e., with matrix variables 
  $Z = \begin{pmatrix} z_{11} & z_{12} \\
    z_{12} & z_{22} \end{pmatrix}$,
  is
    $f(Z) = \det \diag(z_{11},z_{12},z_{22}) = 
    z_{11} z_{12} z_{22}$. Namely, since
    $ \mathcal{I}(f) = \{ X \in \sym_2 \ : \ x_{11} x_{12} x_{22} = 0 \}$,
    we have $\begin{pmatrix} 1 & 0 \\ 0 & 1 \end{pmatrix} \in \mathcal{I}(f)$
    and thus $\mathcal{I}(f) \cap \inter \sym_2^+ \neq \emptyset$.
 }
 \smallskip
 
 \noindent
 c) Another example of a non-psd-stable determinantal polynomial on $2\times 2$-matrices is the determinant of the spectrahedral representation of the open Lorentz cone $g(\mathbf{z})=\det\left(\begin{matrix}
 z_1+z_3 & z_2 \\ 
 z_2 & z_1-z_3
 \end{matrix}
 \right)=z_1^2-z_2^2-z_3^2$, where the same variable identification as in a)
 is used. Note that 
$g(\mathbf{z})=0$ for $\mathbf{z}=(1+2i,1+i,\sqrt{-3+2i})$ and $\left(\begin{matrix}
 2 & 1 \\
 1 & \alpha 
 \end{matrix}\right)\in \inter\sym_2^+$ for $\alpha=\Im(\sqrt{-3+2i})>1$. Hence, 
  $g$ is not psd-stable.
 \end{example}
 
 \iffalse In fact, it is true that determinantal polynomials are always psd-stable \textcolor{blue}{(papri: is it true, or proved for determinantal polynomials in the context of Lax conjecture?}))
\\
\textcolor{green}{What about 
\begin{equation*}
f(x,y,z)=\text{det}\left( \begin{array}{rrr}
3x+y & 2y & 0 \\
2y & 2x+y & 0 \\
0 & 0 & z 
\end{array}\right)=6x^2z+5xyz-3y^2z?
\end{equation*}
For $f$ and $a=2i,b=1/3\cdot(5+\sqrt{97})i, c=50i$ we get $f(a,b,c)=0$, but $Im(a)\cdot Im(c)-Im(b)^2\geq 0, Im(a),Im(c)\geq 0$. So $f$ shouldn't be psd-stable.}

\fi

\section{Conic components in the complement of the imaginary projection\label{se:coniccomponents}}

To prepare for the conic stability criteria 
for determinantal and quadratic polynomials, we characterize 
particular conic components in the complement of 
the imaginary projection for these classes. Denote by $X \succ 0$
the positive definiteness of a matrix $X$.

First consider a determinantal polynomial
\begin{equation}
  \label{eq:fdetpoly}
  f(\mathbf{z}) \ = \ \det(A_0 + A_1 z_1 + \cdots + A_n z_n)
\end{equation}
with $A_0, \ldots, A_n \in \herm_d$. Note that if $A_0=I$, then
the homogenization of $f$ w.r.t.\ a variable $z_0$ is hyperbolic w.r.t.\
$\mathbf{e}=(1,0,\dots,0) \in \R^{n+1}$. Moreover, for a homogeneous determinantal 
polynomial $f=\det(\sum_{j=1}^{n}A_jz_j)$, if there exists an $\e \in \R^{n}$ with 
$\sum_{j=1}^n A_je_j \succ  0$, then $f$ is hyperbolic w.r.t.\ $\e$, and the set
\[
  \{ \mathbf{z} \in \R^{n} \, : \ A_1 z_1 + \cdots + A_n z_n \succ 0 \}
\]
as well as its negative are hyperbolicity cones of $f$, see \cite[Prop.~2]{lpr-2005}. 
If $f$ is irreducible, then these are the only two hyperbolicity cones 
(see \cite{kummer-2019}), whereas in the reducible case there can be 
more (cf.\ Section~\ref{se:prelim-conicstab}).
Let $A(\mathbf{z})$ be the linear matrix pencil 
$A(\mathbf{z}) = A_0 + \sum_{j=1}^n A_j z_j$. The \textit{initial form} of $f$, denoted by $\init(f)$,
is defined as $\init(f)(\z)=f_{h}(0,\z)$, where $f_{h}$ is the homogenization of 
$f$ w.r.t.\ the variable $z_0$. 

\begin{theorem}
\label{th:cones-determinantal}
If $f$ is a degree $d$ determinantal polynomial of the form~\eqref{eq:fdetpoly} and there exists an 
$\mathbf{e} \in \R^n$ with $\sum_{j=1}^n A_j e_j \succ 0$,
then $\init(f)$ is hyperbolic and every hyperbolicity cone of $\init(f)$ is 
contained in $\mathcal{I}(f)^\compl$.
\end{theorem}

\begin{proof}
Let $f = \det(A_0 + \sum_{j=1}^n A_j z_j)$ with 
$A_0, \ldots, A_n \in \herm_d$. Since $f$ is of degree $d$, it holds
$\init(f) = \det(\sum_{j=1}^n A_jz_j)$. Then
$\sum_{j=1}^n A_j e_j \succ 0$ implies that
$\init(f)$ is hyperbolic.

First we assume that $\init(f)$ is irreducible.
By the precondition $\sum_{j=1}^n A_j e_j \succ 0$,
the initial form $\init(f)$ has exactly the two hyperbolicity cones
$C_1 = \{\mathbf{x} \in \R^n \, : \, \sum_{j=1}^n A_j x_j \succ 0\}$
and $C_2 = \{\mathbf{x} \in \R^n \, : \, \sum_{j=1}^n A_j x_j \prec 0\}$.

First we show that $C_1 \subseteq \mathcal{I}(f)^\compl$. 
For every $\mathbf{x} \in \R^n$, we have
\begin{eqnarray*}
  f(\mathbf{x} + t \mathbf{e}) & = & \det(A_0 +  \sum_{j=1}^n A_j x_j 
    + t \sum_{j=1}^n A_j e_j).
\end{eqnarray*}
Since $\sum_{j=1}^n A_j e_j \succ 0$, we obtain
\begin{eqnarray*}
f(\mathbf{x} + t \mathbf{e}) 
  & = & \det(\sum_{j=1}^n A_j e_j) \det\Big( (\sum_{j=1}^n A_j e_j)^{-1/2}
    (A_0 + \sum_{j=1}^n A_j x_j) (\sum_{j=1}^n A_j e_j)^{-1/2}
    + t I \Big) \, .
\end{eqnarray*}
Since $A_0 + \sum_{j=1}^n A_j x_j$ is Hermitian, all the roots of
$t \mapsto f(\mathbf{x} + t \mathbf{e})$ are real.
Hence, there cannot be a non-real vector $\mathbf{a}+i \mathbf{e}$ with 
$f(\mathbf{a}+i\mathbf{e}) = 0$,
because otherwise setting $\mathbf{x} = \mathbf{a}$ would give a non-real solution
to $t \mapsto f(\mathbf{x}+t\mathbf{e})$.
Thus, there is a connected component $C'$ in $\mathcal{I}(f)^{\compl}$ 
containing $C_1$. The case $C_2 \subseteq \mathcal{I}(f)^\compl$ is symmetric,
since $-\mathbf{e} \in C_2$.

To cover also the case of reducible $\init(f)$, 
it suffices to observe that for reducible $\init(f) = \prod_{j=1}^k h_j$ with
irreducible $h_1, \ldots, h_k$, every hyperbolicity cone $C$ of $\init(f)$
is of the form $C = \bigcap_{j=1}^k C_j$ with some hyperbolicity 
cones $C_j$ of $h_j$, $1 \le j \le k$.
\end{proof}

\subsection*{Quadratic polynomials}
Now let $f \in \R[\mathbf{z}]$ be a quadratic polynomial of the form
\begin{equation}
  \label{eq:fquadr}
  f \ = \ \mathbf{z}^T A \mathbf{z} + \mathbf{b}^T \mathbf{z} + c
\end{equation}
with $A \in \sym_n$, $\mathbf{b} \in \R^n$ and $c \in \R$.
We show that those components
of $\mathcal{I}(f)^{\compl}$ which are cones, can be described
in terms of spectrahedra, as made precise in the following.

First recall the situation of a homogeneous quadratic polynomial
$f = \mathbf{z}^T A \mathbf{z}$. By possibly multiplying $A$ with $-1$,
we can assume that the number of positive eigenvalues of $A$ is at least the
number of negative eigenvalues.
In this setting, it is well known that a non-degenerate quadratic form
$f \in \R[\mathbf{z}]$ is hyperbolic if and only if $A$ has signature
$(n-1,1)$ \cite{garding-59}.

Specifically, for the normal form
\[
  f(\mathbf{z}) \ = \ \sum_{j=1}^{n-1} z_j^2 - z_n^2, \,
\]
we have
$
    \mathcal{I}(f) \ = \
    \{\mathbf{y} \in \R^n \; : \, y_n^2\leq  \sum_{j=1}^{n-1} y_j^2\} 
$
(see \cite{jtw-2019}). Hence, there are two
unbounded components in the complement $\mathcal{I}(f)^{\compl}$, 
both of which are full-dimensional cones, and these two components are
\[
\{\mathbf{y} \in \R^{n-1} \times \R_+ \, : \, \sum_{j=1}^{n-1} y_j^2 < y_n^2 \} 
\: \text{ and } \:
  \{\mathbf{y} \in \R^{n-1} \times \R_- \, : \, \sum_{j=1}^{n-1} y_j^2 < y_n^2 \}.
\]
For a general homogeneous quadratic form, this generalizes as follows.
  
\begin{lemma}
\label{le:complquadr}
For a quadratic form $f=\mathbf{z}^T A \mathbf{z} \in \R[\mathbf{z}]$ 
with $A$ having signature $(n-1,1)$, the
components $C$ of the complement of $\mathcal{I}(f)$ are
given by the two components of the set
\begin{equation}
  \label{eq:compl-spectra}
  \{ \mathbf{y} \in \R^{n} \ : \ \mathbf{y}^T A \mathbf{y} < 0 \} \, ,
\end{equation}
and the closures of these components are spectrahedra. 
\end{lemma}

The proof makes use of the following property from~\cite{jtw-2019}.

\begin{proposition}\label{pr:transf2}
	Let $g\in\C[\mathbf{z}]$ and $T\in\R^{n\times n}$ be an invertible matrix. Then,
	$\mathcal{I}(g(T\mathbf{z}))=T^{-1} \mathcal{I}(g(\mathbf{z}))$.
\end{proposition}

\begin{proof}[Proof of Lemma~\ref{le:complquadr}]
Since $-A$ has Lorentzian signature, there exists $S \in \GL(n,\R)$ with
$A_I \ := \ S^T A S \ = \ \diag(1, \ldots, 1, -1)$. Observing
\[
\mathcal{I}(f(S\mathbf{z})) \ = \ \mathcal{I}(\mathbf{z}^TA_I\mathbf{z})
\ = \ \{\mathbf{y}\in\R^n \ : \ y_n^2\leq \sum_{j=1}^{n-1} y_j^2 \}
\ =  \{\mathbf{y}\in\R^n \ : \ \mathbf{y}^TA_I\mathbf{y}\geq 0\} \, ,
\]
Proposition~\ref{pr:transf2} then gives
\[
\mathcal{I}(f(\mathbf{z})) \ = \ S\cdot \mathcal{I}(f(S\mathbf{z}))
\ = \ \{S\cdot \mathbf{y}\in\R^n \ : \ \mathbf{y}^TA_I\mathbf{y}\geq 0\}
\ = \ \{\mathbf{y}\in\R^n \ : \ \mathbf{y}^TA\mathbf{y}\geq 0\} \, .
\]
\end{proof}

For the general, not necessarily homogeneous case, recall that 
every quadric in $\R^n$ is affinely equivalent to a quadric given by 
one of the following polynomials,
\[
  \begin{array}{cll}
  \text{(I)}& \sum_{j=1}^p z_j^2 - \sum_{j=p+1}^r z_j^2 & \quad (1 \le p \le r, \, r \ge 1, \, p \ge \frac{r}{2}) \, , \\ [0.5ex]
 \text{(II)} & \sum_{j=1}^p z_j^2 - \sum_{j=p+1}^r z_j^2 + 1 &  \quad (0 \le p \le r, \, r \ge 1) \, , \\ [0.5ex]
 \text{(III)} & \sum_{j=1}^p z_j^2 - \sum_{j=p+1}^r z_j^2 + z_{r+1} & \quad (1 \le p \le r, \, r \ge 1,
  \, p \ge \frac{r}{2}) \, .
\end{array}
\]
We refer to \cite{berger-book}
as a general background reference for real quadrics.
We say that a given quadratic polynomial $f \in \R[\mathbf{z}]$ is of type $X$ 
if it can be transformed to the normal form $X$ by an affine real transformation.

The homogeneous case, case (I), has already been treated, and by
\cite{jtw-2019}, it is known that in case (III), the imaginary
projection does not contain a full-dimensional component in 
$\mathcal{I}(f)^{\compl}$.

By~\cite{jtw-2019}, in case (II), unbounded components only exist
in the cases $p=1$ and $p=r-1$, so we can restrict to these
cases. We list these relevant two cases from \cite{jtw-2019}.

 \begin{theorem}\label{th:quadrics}
Let $n \ge r \ge 3$ and $f \in \R[\mathbf{z}]$ be a quadratic polynomial.
If $f$ is of type $\mathrm{(II)}$, then
  \begin{equation}\label{eq:class2}
  \mathcal{I}(f) \ = \ \begin{cases}
  \{\mathbf{y} \in \R^n \; : \, y_1^2 - \sum_{j=2}^{r} y_j^2 \le 1\} & \text{if } p = 1 \, , \\
   \{\mathbf{y} \in \R^n \; : \, \sum_{j=1}^{r-1} y_j^2 > y_r^2 \} \cup \{\mathbf{0}\} & \text{if } p = r-1 \, . \\
  \end{cases}
\end{equation}
\end{theorem}

For the proof see \cite{jtw-2019}. 
Since the proofs of the case $p=1$ and of the case $p=r-1$ differ in some important details, which are not carried out there, we include a proof here for the convenience of the reader.

\begin{proof}Without loss of generality we can assume $r=n$.
Writing $z_j=x_j+iy_j$, 
we have $f(\mathbf{z})=\sum_{j=1}^{p}z_j^2- \sum_{j=p+1}^{n}z_j^2+1=0$
if and only if 
\begin{eqnarray}
\sum_{j=1}^{p}x_j^2-\sum_{j=p+1}^{n}x_j^2-\sum_{j=1}^{p}y_j^2+\sum_{j=p+1}^{n}y_j^2+1& = & 0 \label{eq:imclass1} \\
\text{ and } \sum_{j=1}^{p}x_jy_j-\sum_{j=p+1}^{n}x_jy_j & = & 0.
\label{eq:imclass2}
\end{eqnarray}
Set $\alpha:=-\sum_{j=1}^{p}y_j^2+\sum_{j=p+1}^{n}y_j^2+1$,
and let $\mathbf{y} \in \R^n$ be fixed. Note that in both cases
$p=1$ and $p=n-1$, we have ${\bf 0} \in\I(f)$,
since $f(\mathbf{x} + i \cdot \mathbf{0})=0$ for 
$\mathbf{x}=(0, \ldots ,0,1)$. Hence, we can assume ${\bf y} \neq \mathbf{0}$.
 
\medskip

\noindent
\emph{Case $p=1$}: Write 
$\mathbf{x} = (x_1,\mathbf{x}') = (x_1, x_2, \ldots, x_n)$
and $\mathbf{y} = (y_1, \mathbf{y}') = (y_1, y_2, \ldots, y_n)$. Observe the
rotational symmetry of~\eqref{eq:imclass1} w.r.t.\ $\mathbf{x}'$ and 
$\mathbf{y}'$ and
the invariance of the standard scalar product 
$(\mathbf{x}',\mathbf{y}') \mapsto \sum_{j=2}^n x_j y_j$
under orthogonal transformations. Hence, 
if $((x_1,\mathbf{x}'),(y_1,\mathbf{y}'))$
is a solution of~\eqref{eq:imclass1} and~\eqref{eq:imclass2}, then
for any $T \in \SO(n{-}1)$, the point
$((x_1,T\mathbf{x}'),(y_1,T\mathbf{y}'))$ is a solution as well,
where $\SO(n{-}1)$ denotes the special orthogonal group of order $n-1$. 
Thus, we can assume $y_3 = \cdots = y_n = 0$, and $\alpha$ simplifies
to $\alpha = -y_1^2 + y_2^2+1$. Solving~\eqref{eq:imclass2} for $x_1$
(by assuming, without loss of generality, $y_1 \neq 0$)
yields $x_1 = \frac{x_2 y_2}{y_1}$ and substituting this into~\eqref{eq:imclass1}
then 
\[
  0 \ = \ \left( \frac{y_2^2}{y_1^2} -1 \right) 
    x_2^2 - \sum_{j=3}^n x_j^2 + \alpha
     \ = \ \frac{(\alpha - 1) x_2^2}{y_1^2} - \sum_{j=3}^n x_j^2 + \alpha.
\]
This equation has a real solution $(x_2, \ldots, x_n)$ if and only if 
$\alpha \ge 0$, which shows
$\mathcal{I}(f) \ = \{\mathbf{y} \in \R^n \; : \, y_1^2 - \sum_{j=2}^{n} y_j^2 \le 1\}$.

\smallskip

\noindent
\emph{Case $p=n-1$}: Following the same proof strategy, we now
write $x = (\mathbf{x}',x_n) = (x_1, x_2, \ldots, x_n)$
and $y = (\mathbf{y}',y_n) = (y_1, y_2, \ldots, y_n)$.
Then the symmetry of the problem allows to 
assume $y_2 = \cdots = y_{n-1} = 0$, and $\alpha$ simplifies
to $\alpha = -y_1^2 + y_n^2+1$. If $y_1 \neq 0$,
solving~\eqref{eq:imclass2} for $x_1$
gives $x_1 = \frac{x_n y_n}{y_1}$, and a substitution into~\eqref{eq:imclass1}
\[
  0 \ = \ \left( \frac{y_n^2}{y_1^2} -1 \right) 
     x_n^2 + \sum_{j=2}^{n-1} x_j^2 + \alpha 
     \ = \ \frac{(\alpha - 1) x_2^2}{y_1^2} + \sum_{j=2}^{n-1} x_j^2 + \alpha.
\]
There exists a real solution $(x_2, \ldots, x_n)$ if and only if 
$\alpha < 1$, which, taking also into account the special case $y_1=0$, gives $\mathcal{I}(f) = 
\{\mathbf{y} \in \R^n \; : \, \sum_{j=1}^{n-1} y_j^2 > y_n^2 \} \cup \{\mathbf{0}\} $.
\end{proof}

For the inhomogeneous case, we use the following lemma to reduce
it to the homogeneous case.

\begin{lemma}
Let $n \ge 3$ and 
$f \in \R[\mathbf{z}]$ be quadratic of the form~\eqref{eq:fquadr}.

If $f$ is of type~(II) with $p=1$, then $\mathcal{I}(f)^{\compl}$ 
does not have connected components whose closures contain 
full-dimensional cones.

If $f$ is of type~(II) with $p=n-1$ then every full-dimensional cone
which is contained in $\mathcal{I}(f)^\compl$ is contained 
in the closure of a hyperbolicity cone of $\init(f)$. 
\end{lemma}

Note, that in particular, that $\mathcal{I}(f)^{\compl}$ does not
contain a point at all
if and only if $\init(f)$ is not hyperbolic.

\begin{proof}
If $f$ is of type~(II) with $p=1$, 
then the statement is a consequence of~\eqref{eq:class2}.

Now consider the case that $f$ is of type~(II) with $p=n-1$ and let 
$C$ be full-dimensional cone which is contained in
a component of $\mathcal{I}(f)^\compl$. 
By \cite[Theorem 4.2 and Lemma 4.3]{joergens-theobald-hyperbolicity},
$\inter C$ is contained in a hyperbolicity cone of $\init(f)$.
\end{proof}

Hence, among the quadratic polynomials of type (II), only the ones with
$p=n-1$ might possibly be $K$-stable.

\begin{theorem}
\label{th:conic-comp-quadr}
Let $n \ge 3$ and
$f \in \R[\mathbf{z}]$ be quadratic of the form~\eqref{eq:fquadr}
and of type (II) with $p=n-1$.
Then there exists a linear form $\ell(\mathbf{z})$ in $\mathbf{z}$ such that
$-\ell(\mathbf{z})^{n-2} {\init(f)}$ 
has a determinantal representation. In particular, the closure of 
each unbounded component of $\mathcal{I}(f)^\compl$ is a spectrahedral 
cone.
\end{theorem}

The theorem can be seen as an adaption of the well-known result that hyperbolic
quadratic forms have determinantal representations. See, e.g., \cite[Section~2]{vandenberghe-boyd-survey}
or \cite[Example 2.16]{netzer-thom-2012} for the determinantal representations which underlie that
result and which are utilized in the subsequent proof.

\begin{proof}
First consider the normal form of type (II) with $p=n-1$,
\[
  g(\mathbf{z}) \ = \ \sum_{j=1}^{n-1} z_j^2 - z_n^2 + 1 \, .
\]
By~\eqref{eq:class2}, the complement of $\mathcal{I}(g)$ has the
two unbounded conic components
\[
  \{ \mathbf{y} \in \R^{n-1} \times \R_+ \, : \, \sum_{j=1}^{n-1} y_j^2 \le y_n^2 \} \setminus \{0\} \: \text{ and } \:
  \{ \mathbf{y} \in \R^{n-1} \times \R_- \, : \, \sum_{j=1}^{n-1} y_j^2 \le y_n^2 \} \setminus \{0\},
\]
which (up to the origin) are the open Lorentz cone and its negative.
Their closures are exactly the 
closures of the hyperbolicity cones of the initial form $\init(g)$ of $g$.
It is well-known that the open Lorentz cone 
has the spectrahedral representation
\begin{equation}
  \label{eq:lz}
  L(\mathbf{z}) \ := \ \left(
  \begin{array}{ccc|c}
   &          & & z_ 1 \\
   & z_n I&  & \vdots \\
   &         & & z_{n-1} \\ \hline
   z_1 & \cdots & z_{n-1} & z_n
   \end{array} \right) \ \succ \ 0 \, ,
\end{equation}
and thus we also have $z_n^{n-2}\init(g) = -\det(L(\mathbf{z}))$. 
Since $g$ results from $f$ by an affine transformation, the initial 
form $\init(g)$ results from the initial form $\init(f)$ by a linear
transformation,
\[
  \init(g)(T\mathbf{z}) \ = \ \init(f)(\mathbf{z})
\]
for some matrix $T \in \GL(n,\R)$. Hence, we obtain the spectrahedral
representation for one of the unbounded conic components in
$\mathcal{I}(f)^{\compl}$,
\[
  F(\mathbf{z}) \ := \ \left(
  \begin{array}{ccc|c}
   &          & & (T\mathbf{z})_ 1 \\
   & (T \mathbf{z})_n I&  & \vdots \\
   &         & & (T\mathbf{z})_{n-1} \\ \hline
   (T\mathbf{z})_1 & \cdots & (T\mathbf{z})_{n-1} & (T\mathbf{z})_n
   \end{array} \right) \ \succ \ 0 \, ,
\]
as well as its negative.
Moreover,
\[
  -\det F(\mathbf{z}) \ = \ ((T \mathbf{z})_n)^{n-2} \init(f) \, ,
\] 
so that $(T \mathbf{z})_n$ provides the desired linear form $\ell(\mathbf{z})$.
\end{proof}

\begin{remark}\label{re:ldl}
Concerning $L(\mathbf{z})$ in~\eqref{eq:lz}, 
by subtracting $\frac{z_j}{z_n}$ times the $j$-th row from its
$n$-th row for every $j\in\{1,\dots,n-1\}$, we obtain
	\[
        \det(L(\mathbf{z})) = 
        \det\left(
	\begin{array}{ccc|c}
	&          & & z_ 1 \\
	& z_n I&  & \vdots \\
	&         & & z_{n-1} \\ \hline
	0 & \cdots & 0 & z_n-\frac{1}{z_n}\sum\nolimits_{i=1}^{n-1}z_i^2
	\end{array} \right) 
	 =  z_n^{n-2}\left(z_n^2-\sum\limits_{i=1}^{n-1}z_i^2\right).
	\]
By~\eqref{eq:fquadr}, in the proof we have 
$\init(f) = \mathbf{z}^T A \mathbf{z}$. Let $A = LDL^T$ be an
$LDL^T$ decomposition of $A$ with 
$D = \diag(d_1, \ldots, d_{n-1}, d_n)$ such that $d_1, \ldots, d_{n-1} > 0$
and $d_n < 0$. Then the variable transformation $T$ in the proof
is
\[
  T \ = \ \diag( \sqrt{d_1}, \ldots, \sqrt{d_{n-1}}, \sqrt{|d_n|}) \cdot L^T
\]
and we derive 
\[
	A=T^T\cdot 
	\left(
	\begin{array}{ccc|c}
	&          & & 0 \\
	&  I&  & \vdots \\
	&         & & 0 \\ \hline
	0 & \cdots & 0 & -1
	\end{array} \right) \cdot T.
\]
\end{remark}

\begin{ex}
	Consider $f(z_1,z_2,z_3,z_4)=-15z_1^2-12z_1z_4+z_2^2+z_3^2=\mathbf{z}^TA\mathbf{z}$ with
	\begin{equation*}
	A=\left(\begin{matrix}
	-15 & 0 & 0 & -6 \\
	0 & 1 & 0 & 0 \\
	0 & 0 & 1 & 0 \\
	-6 & 0 & 0 & 0
	\end{matrix}\right).
	\end{equation*}
For $\ell(\mathbf{z})=4z_1+2z_4$, 
a representation from Theorem~\ref{th:conic-comp-quadr} is
	\begin{equation*}
	-\ell(\mathbf{z})^2\cdot f(\mathbf{z})=\det \left(\begin{matrix}
	4z_1+2z_4 & 0 & 0 & z_1+2z_4 \\
	0 & 4z_1+2z_4 & 0 & z_2 \\
	0 & 0  & 4z_1+2z_4 & z_3 \\
	z_1+2z_4 & z_2 & z_3 & 4z_1+2z_4
	\end{matrix}\right).
	\end{equation*}
\end{ex}
\begin{remark}
A quadratic polynomial $f \in \R[\z]$ is of the form~\eqref{eq:fquadr}
and of type (II) with $p=n-1$ (i.e., $-f$ has Lorentzian signature) if
and only if $f \in \R[\z]$ is a real zero polynomial, see 
for example \cite{dey-pillai-2018}. 
\end{remark}

\begin{remark}For the case of homogeneous polynomials,
Theorem \ref{th:conic-comp-quadr} recovers the known fact that 
hyperbolicity cones defined by homogeneous quadratic polynomials $f$ are 
spectrahedral \cite{netzer-thom-2012}.
In the affine setting, we can homogenize the type (II) polynomial
$f$ w.r.t.\ variable $z_0$ and get a quadratic polynomial of type (I) in $n+1$ variables with $p=n$. Then, using $\init(f_{h})=f_{h}$, Theorem~\ref{th:conic-comp-quadr} recovers that
\textit{the rigidly convex sets} (introduced by Helton-Vinnikov \cite{helton-vinnikov-2007}) defined by real zero polynomials $f$ are 
spectrahedra \cite{netzer-thom-2012}.
\end{remark}

\begin{remark}
 The proof of Theorem \ref{th:conic-comp-quadr} explicitly explains a technique to compute a suitable linear factor $\ell(\z)$ as well as a determinantal representation to get a spectrahedral structure.
\end{remark}

\section{Conic stability and positive maps\label{se:conicstab-posmaps}}

Based on the characterizations of the conic components in the
complement of $\mathcal{I}(f)$, we now study the problem whether
$f$ is $K$-stable, in particular, whether it is psd-stable.

In order to decide whether the cone $K$ is contained in one of the 
components of $\mathcal{I}(f)^{\compl}$, observe that in the case of spectrahedral
representations of $K$ and of the components of $\mathcal{I}(f)^{\compl}$,
the problem of $K$-stability can be phrased as a containment problem
for spectrahedra. The theory of positive and completely positive maps (as detailed
in \cite{paulsen-2003})
provides a sufficient condition for the containment problem of spectrahedra,
see \cite{hkm-2010,ktt-2013,ktt-2015}.

\begin{definition}
Given two linear subspaces $\mathcal{U}\subseteq \herm_k$ and
$\mathcal{V}\subseteq\herm_l$ 
(or $\mathcal{U}\subseteq \sym_k$ and $\mathcal{V}\subseteq\sym_l$), 
a linear map 
$\Phi:\ \mathcal{U}\rightarrow\mathcal{V}$ is
called \emph{positive} if 
$\Phi(U) \succeq 0$ for any $U \in \mathcal{U}$ with $U \succeq 0$.

For $d \ge 1$, define the \emph{$d$-multiplicity map} $\Phi_d$ on the set of all
Hermitian $d \times d$ block matrices with symmetric $n \times n$-matrix entries
by
\[
  (A_{ij})_{i,j=1}^d \ \mapsto \ \big(\Phi \left(A_{ij}\right) \big)_{i,j=1}^d.
\]
The map $\Phi$ is called \emph{$d$-positive} if the $d$-multiplicity map $\Phi_d$ (viewed
as a map on a Hermitian matrix space) is a positive map.
$\Phi$ is called \emph{completely positive} if $\Phi_d$ is a positive map for all $d \ge 1$.
\end{definition}

Let $U(\mathbf{x}) = \sum_{j=1}^n U_j x_j$ 
and $V(\mathbf{x}) = \sum_{j=1}^n V_j x_j$ be homogeneous linear 
pencils with symmetric matrices of size $k \times k$ and 
$l \times l$, respectively
(since the matrices are symmetric, we prefer to denote the variables by 
$\mathbf{x}$ rather than $\mathbf{z}$).
Then the spectrahedra
$S_U := \{ \x \in \R^n \, : \, U(\mathbf{x}) \succeq 0 \}$,
and $S_V := \{ \x \in \R^n \, : \, V(\mathbf{x}) \succeq 0 \}$
are cones.
Further, let $\mathcal{U} = \myspan(U_1, \ldots, U_n) \subseteq \sym_k$
and
$\mathcal{V} = \myspan(V_1, \ldots, V_n) \subseteq \sym_l$.

If $U_1, \ldots, U_n$ are linearly independent, then
the linear mapping
$\Phi_{UV} : \mathcal{U} \to \mathcal{V}$,
$\Phi_{UV}(U_i) := V_i$, $1 \le i \le n$, is well defined.

\begin{proposition}[\cite{ktt-2013}]
\label{pr:posmaps}
Let  $U_1, \ldots, U_n \subseteq \herm_k$ (or, $U_1, \ldots, U_n \subseteq \sym_k$, respectively)
be linearly independent and $S_U \neq \emptyset$. Then for the properties
\begin{enumerate}
\item the semidefinite feasibility problem
\begin{equation}
  \label{eq:sdfp}
  C = \left( C_{ij} \right)_{i,j=1}^{k} \succeq 0
  \text{ and }
  V_{p} = \sum_{i,j=1}^{k} (U_p)_{ij} C_{ij}
  \text{ for } p = 1,\ldots,n
\end{equation}
has a solution with Hermitian (respectively symmetric) matrix $C$,
\item $\Phi_{UV}$ is completely positive,
\item $\Phi_{UV}$ is positive,
\item $S_U \subseteq S_V$,
\end{enumerate}
the implications and equivalences
  $(1) \ \Longrightarrow \ (2) \ \Longrightarrow (3) \Longleftrightarrow (4)$
hold, and if $\mathcal{U}$ contains a positive definite matrix, $(1) \Longleftrightarrow (2)$.
\end{proposition}

Note that the statement $(1) \Longrightarrow (4)$ (which does not involve
the definition of $\Phi_{UV}$) is also valid without the assumption
of linear independence of $U_1, \ldots, U_n$ (see \cite{hkm-2010,ktt-2013}).

So, in case the cone $K$ and the conic components of 
$\mathcal{I}(f)^{\compl}$ can be described in terms of spectrahedra,
we can approach the conic stability problem in terms of the 
block matrix $C\succeq 0$ in~\eqref{eq:sdfp}, the so-called \emph{Choi matrix},
corresponding to an appropriate positive map $\Phi$, which maps the underlying 
pencils of those spectrahedra onto each other certifying their containment. 
This sufficient condition is provided by a certain semidefinite feasibility problem 
whose non-emptiness of its feasible domain thus provides a sufficient criterion 
for psd-stability.

Moreover, if we know a spectrahedral description of some 
of the components of $\mathcal{I}(f)^{\compl}$
(as in the quadratic case or the determinantal case),
the sufficient containment criterion is based on writing 
a matrix pencil for these components using linear combinations
of the matrices of a linear matrix pencil for $K$.
As formalized in Theorem~\ref{th:quadr-certificate} and 
Corollary~\ref{co:quadr-certificate},
taking the determinant of a matrix pencil for a suitable component
of $\mathcal{I}(f)^{\compl}$ provides a 
particular determinantal description
for the homogeneous part of the given polynomial $f$.
That description has exactly the structure of the sufficient
determinantal criterion for psd-stability and thus provides an elegant
determinantal representation that certifies the psd-stability of 
a homogeneous polynomial $f$.

Let $K$ be a cone which is given as the positive semidefiniteness region 
of a linear matrix pencil
$M(\mathbf{x}) = \sum_{j=1}^n M_j x_j$ with symmetric $l \times l$-matrices
(since $K$ is a cone in $\R^n$, we prefer to denote the variables by $\mathbf{x}$
rather than $\mathbf{z}$).
In the case of usual stability, the cone $K$ is the positive semidefiniteness
region of the linear matrix pencil
\begin{equation}
  \label{eq:kge0}
  M^{\ge 0}(\mathbf{x}) \ = \ \sum_{j=1}^n M^{\ge 0}_j x_j
\end{equation}
with $M^{\ge 0}_j = E_{jj}$, where $E_{ij}$ is the matrix
with a one in position $(i,j)$ and zeros elsewhere.
In the case of psd-stability, the matrix pencil is
\begin{equation}
  \label{eq:kpsd}
  M^{\psdtext}(X) \ = \ \sum_{i,j=1}^{n} M^{\psdtext}_{ij} x_{ij}
\end{equation}
with symmetric matrix variables $X = (x_{ij})$
and $M^{\psdtext}_{ij} = \frac{1}{2}(E_{ij} + E_{ji})$, i.e.,
$M^{\psdtext}(X)$ is the matrix pencil
$M^{\psdtext}(X) = (x_{ij})_{ij}$ in the symmetric matrix variables
$x_{ij}$.

\begin{theorem}
\label{th:kstable-determinantal}
Let $f = \det(A_0 + \sum_{j=1}^n A_j z_j)$ with Hermitian matrices $A_0, \ldots, A_n$
be a degree $d$ determinantal polynomial 
of the form~\eqref{eq:fdetpoly} such that $\init(f)$ is irreducible and 
there exists $\mathbf{e} \in \R^n$ with $\sum_{j=1}^n A_j e_j \succ 0$.
Let $M(\mathbf{x}) = \sum_{j=1}^n M_j x_j$ with symmetric $l \times l$-matrices
be a pencil of the cone $K$.
If there exists a Hermitian block matrix $C=(C_{ij})_{i,j=1}^l$ with 
blocks $C_{ij}$ of size $d \times d$ and

\begin{equation}
  \label{eq:containment-crit1}
  C = (C_{ij})_{i,j=1}^l \ \succeq \ 0,
  \quad \forall p = 1, \ldots, n \; : \;
  \sigma A_p \ = \ \sum_{i,j=1}^l (M_p)_{ij} C_{ij}
\end{equation}
for some $\sigma \in \{-1,1\}$,
then $f$ is $K$-stable. Deciding whether such a block matrix $C$ exists
is a semidefinite feasibility problem.
\end{theorem}

Note that a necessary condition of $K$-stability of $f$ is obtained as follows.
Fix any vector $\mathbf{v}$ in the interior of the cone $K$. Then a necessary condition for
$K$-stability is that $\mathbf{v}$ is contained in the complement of
$\mathcal{I}(f)$.

\begin{proof}Let $C$ be a block matrix $C=(C_{ij})_{i,j=1}^l$ with
$d \times d$-blocks and which satisfies~\eqref{eq:containment-crit1}
for some $\sigma \in \{-1,1\}$.
The initial form $\init(f)$ is hyperbolic and,
by~Theorem~\ref{th:cones-determinantal}, every hyperbolicity
cone of $\init(f)$ is contained in $\mathcal{I}(f)^{\compl}$.
So, in order to show $K$-stability of $f$, it suffices to show that 
$K$ is contained in the closure of a hyperbolicity cone of $\init(f)$, i.e., 
in the closure of a component of $\mathcal{I}(\init(f))^{\compl}$.

As recorded at the beginning of Section~\ref{se:coniccomponents},
since $\init(f)$ is irreducible, $\init(f)$ has exactly 
two hyperbolicity cones, and these are given by
$A^h(\mathbf{x}) = \sum_{j=1}^n A_j x_j \succ 0$ as well as
$A^h(\mathbf{x}) = \sum_{j=1}^n A_j x_j \prec 0$.

By Proposition~\ref{pr:posmaps},
if~\eqref{eq:containment-crit1} is satisfied, say with $\sigma = 1$,
then the spectrahedron given by the matrix pencil $M(\mathbf{x})$ is contained 
in the closure of $\mathcal{I}(\init(f))^{\compl}$. For the service of the reader, 
we provide an explicit derivation of this step in our setting. Namely,
for $\mathbf{x}$ in the spectrahedron defined by $M(\mathbf{x})$, we have
\begin{eqnarray}
	A^h(\mathbf{x}) & = & \sum_{p=1}^n A_p x_p
	\ = \ \sum_{p=1}^n x_p \sum_{i,j=1}^l (M_p)_{ij} C_{ij} 
	\label{eq:comb1} \\
	& = & \sum_{i,j=1}^l (M(\mathbf{x}))_{ij} C_{ij}. \label{eq:comp2}
\end{eqnarray}
Apply the Khatri-Rao product (where the blocks of $M(\mathbf{x})$ are of size $1 \times 1$
and the blocks of $C$ are of size $d \times d$).
Since $M(\mathbf{x})$ and $C$ are positive semidefinite, the Khatri-Rao product
\[
M(\mathbf{x}) * C \ := \ ((M(\mathbf{x}))_{ij} \otimes C_{ij})_{i,j=1}^l
\ = \ ((M(\mathbf{x}))_{ij} C_{ij})_{i,j=1}^l
\]
is positive semidefinite as well; see Liu \cite{liu1999}, where this property is stated
on the space of
symmetric positive semidefinite matrices. 
Since $M(\mathbf{x})$ is a real symmetric pencil, Liu's result carries over to our 
situation of a Hermitian positive semidefinite matrix $C$ by employing that a Hermitian matrix $Z=X+iY$ 
with $X \in \sym_k$ and $Y$ skew-symmetric
is positive semidefinite if and only if the real symmetric matrix 
\[
 \left( \begin{array}{cc}X & -Y \\
  Y & Z 
  \end{array} \right) \ \in \ \sym_{2k}
\]
is positive semidefinite (see, e.g., \cite{goemans-williamson-2004}).

Altogether, since 
\[
A^h(\mathbf{x}) \ = \ (I \cdots I) (M(\mathbf{x})*C) \begin{pmatrix} I \\ \vdots \\ I \end{pmatrix} \, ,
\]
$A^h(\mathbf{x})$ is positive semidefinite
as well. Hence, $\mathbf{x}$ is contained in the spectrahedron
defined by $A^h(\mathbf{x})$.
Since $A^h(\mathbf{x})$ is the matrix pencil of the closure of a component of 
$\mathcal{I}(\init(f))^{\compl}$, the claim follows.
\end{proof}

Note that the constant coefficient matrix
$A_{0}$ does not play any role for the criterion in 
Theorem~\ref{th:kstable-determinantal}. This comes from
Theorem~\ref{th:cones-determinantal} and its proof, where
only the Hermitian property of $A_0$ matters rather than
the exact values of the coefficients themselves.

\begin{remark}\label{rem:usual-stability}
In the special case of usual stability, 
Theorem~\ref{th:kstable-determinantal} provides a new proof for
Borcea and Br\"and\'{e}n's determinantal criterion from Proposition~\ref{pr:crit-stable}.
Namely, for usual stability, $K$ is given by~\eqref{eq:kge0} and thus,
a matrix $C$ satisfying the hypothesis of 
Theorem~\ref{th:kstable-determinantal} can be viewed as a 
a block diagonal matrix $C=(C_{ij})_{i=1}^l$ with 
diagonal blocks $C_{ii}$ of size $d \times d$ and 
vanishing non-diagonal blocks $C_{ij}$ ($i \neq j$).
Since the condition~\eqref{eq:containment-crit1} specializes
to 
\[
  A_p = C_{pp} \quad \text{ for } p = 1, \ldots, n,
\]
the stability criterion in Theorem~\ref{th:kstable-determinantal}
is satisfied if and only if the matrices
$A_1, \ldots, A_n$ are positive semidefinite.
\end{remark}

\begin{remark}Theorem~\ref{th:kstable-determinantal} gives a sufficient
criterion, but it is not necessary. As a counterexample,
consider the following adaption from an example 
in~\cite[Example 3.1, 3.4]{hkm-2010}
and~\cite[Section 6.1]{ktt-2013}.
Let $K \subseteq \R^3$ be the Lorentz cone as given by~\eqref{eq:lz}.
The polynomial
\[
  f \ = \ \det \left( \begin{array}{cc}
     z_1 + z_3 & z_2 \\
     z_2 & -z_1 + z_3
    \end{array}
   \right) \ = \ z_3^2-z_1^2-z_2^2
\]
(whose underlying matrix pencil provides an alternative matrix pencil
for the Lorentz cone) has all its zeroes
on the boundary of the Lorentz cone or on its negative. 
Hence, $f$ is $K$-stable,
but by the results in \cite{hkm-2010} and 
\cite{ktt-2013}, the condition~\eqref{eq:containment-crit1}
is not satisfied.
\end{remark}

\begin{example}
\label{ex:example-determinantal}
i) Let $g(z_1,z_2,z_3):=31z_1^2+32z_1z_3+8z_3^2-8z_1z_2-16z_2^2$. A determinantal representation of $g$ is given by $\text{det}\left(\begin{matrix}
4z_1+2z_3 & z_1+4z_2 \\
z_1+4z_2 & 8z_1+4z_3
\end{matrix}\right)$, and at $\mathbf{z}=(0,0,1)^T$, the matrix polynomial
is positive definite.
Let $M(\mathbf{x})$ denote the linear matrix pencil of the psd cone $\sym^+_2$. Then the psd-stability of $g$ follows from Theorem~\ref{th:kstable-determinantal} and by the matrix 
\begin{equation*}
C=\left( \begin{array}{rrrr}
4 & 1 & 0 & 2 \\
1 & 8 & 2 & 0 \\
0 & 2 & 2 & 0 \\
2 & 0 & 0 & 4
\end{array}\right)\succeq 0.
\end{equation*}

ii) Let $f = \sum_{i,j=1}^2 M_{ij}^{\psdtext} x_{ij}
  = \begin{pmatrix} 1 & 0 \\ 0 & 0 \end{pmatrix} x_{11}
  + \begin{pmatrix} 0 & 1 \\ 1 & 0 \end{pmatrix} x_{12}
  + \begin{pmatrix} 0 & 0 \\ 0 & 1 \end{pmatrix} x_{22}$ 
  be the canonical matrix polynomial of the $2 \times 2$-psd cone. Clearly,
  $f$ is psd-stable, and the following consideration shows that this is also
  recognized by the sufficient criterion.
  For symmetric $2 \times 2$-matrices, the condition in 
  Theorem~\ref{th:kstable-determinantal} requires to find 
  a block matrix $C\succeq 0$ with $2 \times 2$ blocks of size $2 \times 2$
  such that
\begin{equation}
  \label{eq:psdinpsd}
 M_{pq}^{\psdtext} =\sum_{i,j=1}^{2} (M_{pq}^{\psdtext})_{ij} C_{ij}
 \quad \text{ for $1 \le p, q \le 2$.}
\end{equation}
This yields $C_{11}=\left(\begin{matrix}
1 & 0 \\
0 & 0
\end{matrix}\right)$,
$C_{22}=\left(\begin{matrix}
0 & 0 \\
0 & 1
\end{matrix}\right)$
and $C_{12}+C_{21}=\left(\begin{matrix}
0 & 1 \\
1 & 0
\end{matrix}\right)$. Since $C = \begin{pmatrix} C_{11} & C_{12} \\ C_{21} & C_{22} \end{pmatrix} $ is symmetric, $C_{12}$ must be of the form
$\begin{pmatrix}
0 & \gamma \\
\delta & 0
\end{pmatrix}$ with $\gamma, \delta \in \R$. Positive semidefiniteness of 
$C$
then implies $\delta = 0$, and further, the condition on $C_{12} + C_{21}$
gives $\gamma = 1$. Hence, the matrix
\[
C=\left(\begin{matrix}
1 & 0 & 0 & 1 \\
0 & 0 & 0 & 0 \\
0 & 0 & 0 & 0 \\
1 & 0 & 0 & 1
\end{matrix}\right)
\]
satisfies~\eqref{eq:psdinpsd} and thus
certifies the psd-stability of $f$ in view of the sufficient 
criterion in Theorem~\ref{th:kstable-determinantal}.
\end{example}

For quadratic polynomials, we can provide the following criterion.
As in the proof of Theorem~\ref{th:conic-comp-quadr},
for a homogeneous quadratic polynomial
$f(\mathbf{z}) = \mathbf{z}^T A \mathbf{z}$ 
of signature $(n-1,1)$,
we consider
\begin{equation}
\label{eq:ft}
  F(\mathbf{x}) := \sum_{p=1}^n F_p x_p \ := \ \left(
  \begin{array}{ccc|c}
   &          & & (T\mathbf{x})_ 1 \\
   & (T \mathbf{x})_n I&  & \vdots \\
   &         & & (T\mathbf{x})_{n-1} \\ \hline
   (T\mathbf{x})_1 & \cdots & (T\mathbf{x})_{n-1} & (T\mathbf{x})_n
   \end{array} \right) \ \succ \ 0 \, ,
\end{equation}
where $T$ is as in that proof.

\begin{theorem}
\label{th:kstable-quadratic}
Let $n \ge 3$ and
$f$ be a quadratic polynomial of the form~\eqref{eq:fquadr},
let $f$ be of type (II) with $A$ having signature $\left(n-1,1\right)$ and $\init(f)$ be
irreducible. 
Let $M(\mathbf{x})$ be a matrix pencil for the cone $K$, and
let $T$ and $F(\mathbf{x}) := \sum_{p=1}^n F_p x_p$ be defined
as in~\eqref{eq:ft} w.r.t.\ $\init(f)$.
If there exists a block matrix $C=(C_{ij})_{i=1}^l$ with 
blocks $C_{ij}$ of size $d \times d$ and
\begin{equation}
  \label{eq:containment-crit2}
  C = (C_{ij})_{i,j=1}^l \ \succeq \ 0,
  \quad \forall p = 1, \ldots, n \; : \;
  \sigma F_p \ = \ \sum_{i,j=1}^l (M_p)_{ij} C_{ij}
\end{equation}
for some $\sigma \in \{-1,1\}$, 
then $f$ is $K$-stable.
Deciding whether such a block matrix $C$ exists is a semidefinite
feasibility problem.
\end{theorem}

\begin{proof}
By~Theorem~\ref{th:conic-comp-quadr} and its proof, the unbounded
components of $\mathcal{I}(f)^\compl$ which are full-dimensional cones
are exactly the hyperbolicity cones of $\init(f)$.
For $\mathbf{x}$ in the spectrahedron defined by 
$M(\mathbf{x})\succeq 0$, we have
	\[
		F(\mathbf{x}) \ = \ \sum_{p=1}^n F_p x_p 
		\ = \ \sum_{p=1}^n x_p \sum_{i,j=1}^l (M_p)_{ij} C_{ij}
		\ = \ \sum_{i,j=1}^l (M(\mathbf{x}))_{ij} C_{ij}.
	\]
Analogous to the application of the Khatri-Rao product in the 
proof of Theorem~\ref{th:kstable-determinantal}, this yields
$F(\mathbf{x}) \succeq 0$. Hence, $f$ is $K$-stable.
\end{proof}

\begin{theorem}
  \label{th:quadr-certificate}
Let $n \ge 3$ and $f(\mathbf{z}) = \mathbf{z}^T A \mathbf{z}$ 
be an irreducible homogeneous quadratic polynomial of signature $(n-1,1)$,
$M(\mathbf{z})$ be a matrix pencil for the cone $K$, and
let $T$ and $F(\mathbf{z}) := \sum_{p=1}^n F_p z_p$ be defined
as in~\eqref{eq:ft}.
If there exists a block matrix $C=(C_{ij})_{i=1}^l$ with
blocks $C_{ij}$ of size $d \times d$ satisfying
\begin{equation}
  \label{eq:containment-cr1b}
  C = (C_{ij})_{i,j=1}^l \ \succeq \ 0,
  \quad \forall p = 1, \ldots, n \; : \;
  \sigma F_p \ = \ \sum_{i,j=1}^l (M_p)_{ij} C_{ij}
\end{equation}
for some $\sigma \in \{-1,1\}$,
then there exists a linear form $\ell(\mathbf{z})$ such that
$-\ell(\mathbf{z})^{n-2} f$ has a determinantal representation
\[
  - \sigma \ell(\mathbf{z})^{n-2} f = \det(\sum_{p=1}^n z_p \sum_{i,j=1}^l (M_p)_{ij} C_{ij})
\]
with positive semidefinite matrices $C_{ij}$.
The representation provides a certificate for the $K$-stability of $f$.
\end{theorem}

\begin{proof}
The $K$-stability was shown in Theorem~\ref{th:kstable-quadratic}.
By~\eqref{eq:containment-cr1b} and the definition of $F(\mathbf{z})$,
we have
\[
  \sigma \det F(\mathbf{z})  =
  \det \big( \sum_{p=1}^n z_p \sum_{i,j=1}^l (M_p)_{ij} C_{ij} \big).
\]
Since $\det F(\mathbf{z}) = - ((T \mathbf{z})_n)^{n-2} f$, 
the choice $\ell(\mathbf{z}) := (T \mathbf{z})_n$
provides the desired representation.
This provides a certificate for the $K$-stability
of $f$.
\end{proof}

\begin{corollary}\label{co:quadr-certificate}
Let $n \ge 2$ and $f(Z)$ be a homogeneous quadratic polynomial on symmetric 
$n \times n$-variables, in the linearized vector $\mathbf{z} = (z_1, \ldots, z_N)$ let
$f = \mathbf{z}^T A \mathbf{z}$ with $A \in \R^{N \times N}$ of signature $(N-1,1)$.
If $M(\mathbf{z})$ is a matrix pencil for the psd-cone and
$C$ is a block matrix satisfying~\eqref{eq:containment-cr1b}, then 
for some linear form $\ell(\mathbf{z})$ in $\mathbf{z}$, the polynomial
$-\ell(\mathbf{z})^{N-2} f$ has a determinantal representation of the
form
\[
  -\ell(\mathbf{z})^{N-2} f \ = \ \det\big( \sum_{i,j=1}^l C_{ij} z_{ij} \big)
\]
with positive semidefinite matrices $C_{ij}$.
This representation provides a certificate for the psd-stability of $f$ in
the sense of the sufficient criterion for psd-stability.
\end{corollary}

\begin{proof}
  This is a consequence of Theorem~\ref{th:quadr-certificate}.
\end{proof}

\section{Certifying $K$-stability with respect to scaled cones\label{se:scaled}}

The sufficient criterion does not capture all the cases
of $K$-stable polynomials. Here, we extend our techniques 
to scaled versions of the cone. To this end, we will 
reduce a scaled version of the $K$-stability problem 
to the situation of the following statement.

\begin{proposition}[Proposition~6.2 in \cite{ktt-2013}]\label{pr:ktt-bounded}
	Let $A(\mathbf{z})$ and $B(\mathbf{z})$ be monic linear matrix pencils of size $k \times k$ and $l \times l$, respectively, and such that $S_A:=\{\mathbf{z}\in\R^n:A(\mathbf{z})\succeq 0 \}$ is bounded. Then there exists a constant $\nu >0$ such that for the scaled spectrahedron $\nu S_A$ the inclusion $\nu S_A\subseteq S_B$ is certified by the system
	\begin{equation*}
	C = (C_{ij})_{i,j=1}^k \ \succeq \ 0,
	\quad \forall p = 1, \ldots, n \; : \;
	B_p \ = \ \sum_{i,j=1}^k \Big( \frac{1}{\nu} A_p \Big)_{ij} C_{ij}.
	\end{equation*}
\end{proposition}

As before, let $K$ be a proper cone which is given by a linear matrix pencil
$M(\mathbf{z}) = \sum_{j=1}^n M_j z_j$ with $l \times l$-matrices,
and assume that there exists a hyperplane 
$H$ not passing through the origin and such that $K \cap H$ is bounded. For 
notational convenience, assume that
$H = \{(x_1, \ldots, x_n) \in \R^{n} \, : \, x_1 = 1\}$ and
that $M_1 = I_n$. In particular, then the first unit vector 
$\mathbf{e}^{(1)}$ is contained in the interior of the full-dimensional
cone $K$.

\begin{theorem}\label{th:nu-cont}
Let $f \in \R[\mathbf{z}]$ and $M(\mathbf{z})$ be as described before.
Let $N(\mathbf{z})$ be the matrix pencil of a spectrahedral, conic 
set contained in $\cl (\mathcal{I}(f)^{\compl})$, and assume that 
$N_1 = I_n$ as well.

Then there exists a constant $\nu > 0$ such that 
$g_\nu(z_1, \ldots, z_n) := f(z_1, \nu z_2, \ldots, \nu z_n)$
is $K$-stable and such that the $K$-stability of $g$ is certified by the
system
\begin{equation}\label{eq:c-system}
  C = (C_{ij})_{i,j=1}^l \succeq 0,
  \quad \forall p = 1, \ldots, n \; : \;
  \nu N_p \ = \ \sum_{i,j=1}^l \left( M_p\right)_{ij} C_{ij},
\end{equation}
where the variable matrix $C$ is a block matrix with $l \times l$ blocks.

As a consequence, $f$ is $\hat{K}$-stable with respect to
$\hat{K} = \cone (\{1\} \times \nu(K \cap H))$, 
where the multiplication of $\nu$ with the set
$K \cap H$ is done in the $(n-1)$-dimensional space with variables 
$\mathbf{z}' = (z_2, \ldots, z_n)$ and $\cone$ denotes the conic hull.
\end{theorem}

Since the scaling variable $\nu$ occurs linearly in~\eqref{eq:c-system},
its optimal value can be expressed by a semidefinite program.
Further note that the preconditions $M_1 = I_n$ and $N_1 = I_n$ imply that 
the induced matrix pencils of the conic spectrahedra of $M(\mathbf{z})$ and
of $N(\mathbf{z})$ give monic pencils within the hyperplane $H$.

\begin{proof}
Let $N'(\mathbf{z}'), M'(\mathbf{z}')$ be the matrix pencils in the $n-1$ variables 
$\mathbf{z}'=(z_2, \ldots, z_n)$
defined by
\[
  N'(\mathbf{z}') = N(\mathbf{z}) \Big|_{z_1 = 1} \text{ and }
  M'(\mathbf{z}') = M(\mathbf{z}) \Big|_{z_1 = 1}.
\]
$N'(\mathbf{z}')$ and $M'(\mathbf{z}')$ are monic linear matrix pencils
and the spectrahedron
$S_{M'(\mathbf{z}')} = \{\mathbf{z}' = (z_2, \ldots, z_n) \in 
\R^n \, : \, M'(\mathbf{z}') \succeq 0\}$ is bounded.
By Proposition~\ref{pr:ktt-bounded}, the inclusion
$ \nu S_{M'(\mathbf{z}')} \subseteq S_{L'(\mathbf{z}')}$ is certified by the system
\begin{equation}\label{eq:c-system2}
  C = (C_{ij})_{i,j=1}^l \ \succeq \ 0,
  \quad \forall p = 1, \ldots, n \; : \;
  N'_p \ = \ \sum_{i,j=1}^l \Big( \frac{1}{\nu} M'_p \Big)_{ij} C_{ij}
\end{equation}
with some block matrix $C =  (C_{ij})_{i,j=1}^l$. Since $M_p' = M_p$
and $N_p' = N_p$ for $p \ge 1$,
this is equivalent to~\eqref{eq:c-system}.

Moreover, $\nu S_{M'(\mathbf{z}')} \subseteq S_{N'(\mathbf{z}')}$ implies that
$ \nu S_{M(\mathbf{z})} \subseteq S_{N(\mathbf{z})}$ and also that
for any $\mathbf{z}$ with $z_1 = 1$ and
$f(\mathbf{z}) = 0$, we have 
$(1, \frac{z_2}{\nu}, \ldots, \frac{z_n}{\nu}) \not\in \inter S_{M'(\mathbf{z}')}$, or,
equivalently, $g_{\nu}(\mathbf{z})$ is $K$-stable. Finally, this also 
gives the reformulation that $f$ is $\hat{K}$-stable.
\end{proof}

Theorem~\ref{th:nu-cont} can also be applied to such polynomials $f$ which meet the requirements of the theorem after applying a invertible linear transformation, since those preserve the containment of sets.

\begin{example}
     Setting $\begin{pmatrix} z_1 & z_2 \\ z_2 & z_3 \end{pmatrix}
     = \begin{pmatrix} z_{11} & z_{12} \\ z_{12} & z_{22} \end{pmatrix}$,
     the polynomial $f=\det \begin{pmatrix} z_1 & 2 z_2 \\ 2 z_2 & z_3 \end{pmatrix} $ $= z_1z_3-4z_2^2$ is not psd-stable. To fit the requirements of Theorem~\ref{th:nu-cont},
     let $Q$ be the rotation matrix 
     $Q=\frac{1}{\sqrt{2}}\left(\begin{matrix}
	1 & 0 & 1 \\
	0 & \sqrt{2} & 0 \\
	-1 & 0 & 1
	\end{matrix}\right)$ and consider the rotated versions of the underlying matrix pencils
	\begin{eqnarray*} 
		N_Q(\mathbf{y}) & = & N(Q^{-1} \mathbf{z}) \ = \ 
		\frac{1}{\sqrt{2}}\left(\begin{matrix} y_1-y_3 & \sqrt{8}y_2 \\
			\sqrt{8}y_2 & y_1+y_3
		\end{matrix}\right) \\
		\text{ and } M_Q(\mathbf{y}) & = & M(Q^{-1} \mathbf{z}) \ = \
		\frac{1}{\sqrt{2}}\left(
		\begin{matrix}
			y_1-y_3 & \sqrt{2}y_2 \\
			\sqrt{2}y_2 & y_1+y_3
		\end{matrix}
		\right).
	\end{eqnarray*}
	For $N_{Q,\nu}(\mathbf{y}):=N_Q(y_1,\nu y_2,\nu y_3)$ and $M_Q(\mathbf{y})$,
	 $\left( \ref{eq:c-system} \right)$ leads to the equations
	\begin{eqnarray*} 
		C_{11} + C_{22} & = & \left(\begin{matrix}
			1 & 0 \\
			0 & 1
		\end{matrix}\right), \quad
		C_{12} + C_{21} \ = \ \left(\begin{matrix}
			0 & 2\nu \\
			2\nu & 0
		\end{matrix}\right), \quad
		-C_{11} + C_{22} = \left(\begin{matrix}
			-\nu & 0\\
			0 & \nu
		\end{matrix}\right).
    \end{eqnarray*}
    Hence, the set of matrices $C = C_{\nu}$ satisfying~\eqref{eq:c-system} is given by
    the system
    \begin{equation}
    \label{eq:cnu}
    C \ = \ \frac{1}{2}\left(
    \begin{matrix}
    1+\nu & 0 & 0 & 4\lambda \nu \\
    0 & 1-\nu & 4(1-\lambda) \nu & 0 \\
    0 & 4(1-\lambda) \nu & 1-\nu & 0 \\
    4 \lambda \nu & 0 & 0 & 1+\nu
    \end{matrix}
    \right), \quad C \succeq 0 \quad \text{with $\lambda \in \R$.}
    \end{equation}
    The largest $\nu$ satisfying~\eqref{eq:cnu} is given by $\nu=\frac{1}{2}$ with 
    $\lambda =\frac{3}{4}$. When rotating back, this certifies the psd-stability of
    \begin{equation*}
    f_\frac{1}{2}(\mathbf{z}) \ := \ \det \left(N_{Q,\frac{1}{2}}(Q\mathbf{z})\right)
    \ = \ \frac{1}{16}\cdot(3z_1^2+10z_1z_3+3z_3^2-16z_2^2).
    \end{equation*}
    In addition to that, we obtain that $f$ is $\hat{K}$-stable with respect 
    to the cone
    \begin{equation*}
    \hat{K} \ = \ \left\{\mathbf{y}\in\R^3 \ : \ \frac{1}{2}\left(
    \begin{matrix}
    3y_1-y_3 & 4y_2 \\
    4y_2 & -y_1+3y_3
    \end{matrix}
    \right)\succeq 0\right\}.
    \end{equation*}
\end{example}

\section{Conclusion and open questions\label{se:conclusion}}

In this paper, we have shown how techniques from the theory of positive
maps and from the containment of spectrahedra can be used to provide
a sufficient criterion for the $K$-stability of a given polynomial $f$.
In particular, we have considered quadratic and determinantal polynomials.
Beyond that, our approach generally applies whenever (for a polynomial of 
arbitrary degree) some spectrahedral components in the complement of 
$\mathcal{I}(f)$ are known. 

It would be interesting to understand whether this 
or related techniques can be effectively exploited also for classes of 
polynomials beyond the ones studied in the paper. In particular, with
regard to the recent development of a theory of Lorentzian 
polynomials \cite{lorentzian-braenden-huh}, which provides a superset
of the set of homogeneous stable polynomials, it would be of interest
to understand the connection of Lorentzian polynomials to conic
stability and to the effective methods presented in our paper.

\bibliography{bibpsdstable}

\begin{thebibliography}{10}

\bibitem{berger-book}
M.~Berger.
\newblock {\em Geometry. {I + II}}.
\newblock Universitext. Springer-Verlag, Berlin, 1987.

\bibitem{borcea-braenden-2008}
J.~Borcea and P.~Br\"{a}nd\'{e}n.
\newblock Applications of stable polynomials to mixed determinants: {J}ohnson's
  conjectures, unimodality, and symmetrized {F}ischer products.
\newblock {\em Duke Math. J.}, 143(2):205--223, 2008.

\bibitem{borcea-braenden-leeyang1}
J.~Borcea and P.~Br{\"a}nd{\'e}n.
\newblock The {L}ee-{Y}ang and {P}{\'o}lya-schur programs. {I}. {L}inear
  operators preserving stability.
\newblock {\em Invent.\ Math.}, 177(3):541, 2009.

\bibitem{borcea-braenden-2010}
J.~Borcea and P.~Br\"{a}nd\'{e}n.
\newblock Multivariate {P}\'{o}lya-{S}chur classification in the {W}eyl
  algebra.
\newblock {\em Proc.\ London Math.\ Soc.}, 101:73--104, 2010.

\bibitem{bbl-2009}
J.~Borcea, P.~Br{\"a}nd{\'e}n, and T.~Liggett.
\newblock Negative dependence and the geometry of polynomials.
\newblock {\em J.\ Amer.\ Math.\ Soc.}, 22(2):521--567, 2009.

\bibitem{braenden-hpp}
P.~Br\"{a}nd\'{e}n.
\newblock Polynomials with the half-plane property and matroid theory.
\newblock {\em Adv. Math.}, 216:302--320, 2007.

\bibitem{braenden-2011}
P.~Br{\"a}nd{\'e}n.
\newblock Obstructions to determinantal representability.
\newblock {\em Adv.\ Math.}, 226(2):1202--1212, 2011.

\bibitem{lorentzian-braenden-huh}
P.~Br\"{a}nd\'{e}n and J.~Huh.
\newblock Lorentzian polynomials.
\newblock {\sf https://arxiv.org/abs/1902.03719}, 2019.

\bibitem{bcs-book}
P.~B{\"u}rgisser, M.~Clausen, and M.~A. Shokrollahi.
\newblock {\em Algebraic Complexity Theory}.
\newblock Springer, Berlin, 1997.

\bibitem{dey-pillai-2018}
P.~Dey and H.~K. Pillai.
\newblock A complete characterization of determinantal quadratic polynomials.
\newblock {\em Linear Algebra Appl.}, 543:106--124, 2018.

\bibitem{garding-59}
L.~G{\aa}rding.
\newblock {An inequality for hyperbolic polynomials}.
\newblock {\em J. Math. Mech.}, 8:957--965, 1959.

\bibitem{goemans-williamson-2004}
M.~X. Goemans and D.~P. Williamson.
\newblock Approximation algorithms for max-3-cut and other problems via complex
  semidefinite programming.
\newblock {\em J.\ Computer \& System Sciences}, 68(2):442--470, 2004.

\bibitem{gkv-2016}
A.~Grinshpan, D.~S. Kaliuzhnyi-Verbovetskyi, V.~Vinnikov, and H.~J. Woerdeman.
\newblock Contractive determinantal representations of stable polynomials on a
  matrix polyball.
\newblock {\em Math.\ Zeitschrift}, 283(1-2):25--37, 2016.

\bibitem{gkv-2017}
A.~Grinshpan, D.~S. Kaliuzhnyi-Verbovetskyi, V.~Vinnikov, and H.~J. Woerdeman.
\newblock Rational inner functions on a square-matrix polyball.
\newblock In {\em Harmonic Analysis, Partial Differential Equations, Banach
  Spaces, and Operator Theory (vol.~2)}, pages 267--277. Springer, Cham, 2017.

\bibitem{lorentzian-hehl-itin}
F.W. Hehl and Y.~Itin.
\newblock Is the {L}orentzian signature of the metric of spacetime
  electromagnetic in origin?
\newblock {\em Ann. of Physics}, 312:60--83, 2004.

\bibitem{hkm-2010}
J.~W. Helton, I.~Klep, and S.~McCullough.
\newblock The matricial relaxation of a linear matrix inequality.
\newblock {\em Math.\ Program.}, 138(1-2, Ser. A):401--445, 2013.

\bibitem{hmv-2006}
J.~W. Helton, S.A. McCullough, and V.~Vinnikov.
\newblock Noncommutative convexity arises from linear matrix inequalities.
\newblock {\em J. Funct. Anal.}, 240(1):105--191, 2006.

\bibitem{helton-vinnikov-2007}
J.~W. Helton and V.~Vinnikov.
\newblock Linear matrix inequality representation of sets.
\newblock {\em Comm. Pure Appl. Math.}, 60:654--674, 2007.

\bibitem{hoermander-1990}
L.~H\"ormander.
\newblock {\em An Introduction to Complex Analysis in Several Variables}.
\newblock North-Holland, 3rd edition, 1990.

\bibitem{joergens-theobald-conic}
T.~J\"{o}rgens and T.~Theobald.
\newblock Conic stability of polynomials.
\newblock {\em Res. Math. Sci.}, 5(2):Paper No. 26, 2018.

\bibitem{joergens-theobald-hyperbolicity}
T.~J{\"o}rgens and T.~Theobald.
\newblock Hyperbolicity cones and imaginary projections.
\newblock {\em Proc.\ Amer.\ Math.\ Soc.}, 146:4105--4116, 2018.

\bibitem{jtw-2019}
T.~J\"{o}rgens, T.~Theobald, and T.~de~Wolff.
\newblock Imaginary projections of polynomials.
\newblock {\em J.\ Symb.\ Comp.}, 91:181--199, 2019.

\bibitem{ktt-2013}
K.~Kellner, T.~Theobald, and C.~Trabandt.
\newblock {Containment problems for polytopes and spectrahedra}.
\newblock {\em SIAM J. Optim.}, 23(2):1000--1020, 2013.

\bibitem{ktt-2015}
K.~Kellner, T.~Theobald, and C.~Trabandt.
\newblock A semidefinite hierarchy for containment of spectrahedra.
\newblock {\em SIAM J. Optim.}, 25(2):1013--1033, 2015.

\bibitem{kummer-2019}
M.~D. Kummer.
\newblock On the connectivity of the hyperbolicity region of irreducible
  polynomials.
\newblock {\em Adv. Geom.}, 19(2):231--233, 2019.

\bibitem{lpr-2005}
A.~S. Lewis, P.~A. Parrilo, and M.~V. Ramana.
\newblock The {L}ax conjecture is true.
\newblock {\em Proc. Amer. Math. Soc.}, 133:2495--2499, 2005.

\bibitem{liu1999}
S.~Liu.
\newblock Matrix results on the {K}hatri-{R}ao and {T}racy-{S}ingh products.
\newblock {\em Linear Algebra Appl.}, 289(1):267 -- 277, 1999.

\bibitem{mss-interlacing1}
A.~W. Marcus, D.~A. Spielman, and N.~Srivastava.
\newblock Interlacing families {I}: {B}ipartite {R}amanujan graphs of all
  degrees.
\newblock {\em Ann.\ Math.}, 182(1):307--325, 2015.

\bibitem{mss-interlacing2}
A.~W. Marcus, D.~A. Spielman, and N.~Srivastava.
\newblock Interlacing families {II}: {M}ixed characteristic polynomials and the
  {K}adison-{S}inger problem.
\newblock {\em Ann.\ Math.}, 182(1):327--350, 2015.

\bibitem{netzer-thom-2012}
T.~Netzer and A.~Thom.
\newblock Polynomials with and without determinantal representations.
\newblock {\em Linear Algebra Appl.}, 437(7):1579--1595, 2012.

\bibitem{paulsen-2003}
V.~Paulsen.
\newblock {\em {Completely Bounded Maps and Operator Algebras}}.
\newblock Cambridge University Press, 2003.

\bibitem{pemantle-2012}
R.~Pemantle.
\newblock Hyperbolicity and stable polynomials in combinatorics and
  probability.
\newblock In D.~Jerison, B.~Mazur, and T.~Mrowka et~al., editors, {\em Current
  Development in Mathematics 2011}, pages 57--124. Int.\ Press, Somerville, MA,
  2012.

\bibitem{piaetetski-shapiro-1969}
I.~I. Piatetski-Shapiro.
\newblock {\em Automorphic Functions and the Geometry of Classical Domains}.
\newblock Gordon and Breach, New York, 1969.

\bibitem{quarez-2012}
R.~Quarez.
\newblock Symmetric determinantal representation of polynomials.
\newblock {\em Linear Algebra Appl.}, 436(9):3642--3660, 2012.

\bibitem{stein-weiss-1971}
E.~M. Stein and G.~Weiss.
\newblock {\em Introduction to Fourier Analysis on Euclidean Spaces}.
\newblock Princeton University Press, 1971.

\bibitem{straszak-vishnoi-2017}
D.~Straszak and N.~K. Vishnoi.
\newblock Real stable polynomials and matroids: {O}ptimization and counting.
\newblock In {\em Proc.\ Symp.\ Theory of Computing, Montreal}. ACM, 2017.

\bibitem{valiant-1979}
L.~G Valiant.
\newblock Completeness classes in algebra.
\newblock In {\em Proc.\ 11th ACM Symposium on Theory of Computing}, pages
  249--261, 1979.

\bibitem{van-der-geer-2008}
G.~van~der Geer.
\newblock Siegel modular forms and their applications.
\newblock In {\em The 1-2-3 of modular forms}, Universitext, pages 181--245.
  Springer, Berlin, 2008.

\bibitem{vandenberghe-boyd-survey}
L.~Vandenberghe and S.~Boyd.
\newblock Semidefinite programming.
\newblock {\em SIAM Review}, 38(1):49--95, 1996.

\bibitem{volcic-2019}
J.~Vol{\v{c}}i{\v{c}}.
\newblock Stable noncommutative polynomials and their determinantal
  representations.
\newblock {\em SIAM J.\ Appl.\ Algebra \& Geometry}, 3(1):152--171, 2019.

\bibitem{wagner-2010}
D.~G. Wagner.
\newblock Multivariate stable polynomials: Theory and applications.
\newblock {\em Bull. Amer. Math. Soc.}, 48:53--84, 2011.

\end{thebibliography}
\bibliographystyle{plain}

\end{document}